\documentclass[]{amsart}
\usepackage{amsfonts}
\usepackage{amssymb}
\usepackage{amsmath}
\usepackage{amsthm}
\usepackage{amscd}
\usepackage{mathrsfs}
\usepackage[all]{xy}
\usepackage{graphicx, color}
\usepackage{url}

\allowdisplaybreaks[0]

\theoremstyle{theorem}
\newtheorem{theorem}{Theorem}[section]
\newtheorem{lemma}[theorem]{Lemma}
\newtheorem{corollary}[theorem]{Corollary}

\theoremstyle{definition}

\newtheorem{example}[theorem]{Example}
\newtheorem{question}[theorem]{Question}
\newtheorem{remark}[theorem]{Remark}

\newcommand{\RotE}[1]{\mathrm{Rot} \hskip 0.1em \mathbb{E}^{#1}}

\begin{document}

\title[Torus knot and PL trochoid]{On colorability of knots by rotations, \\ Torus knot and PL trochoid}
\author[Ayumu Inoue]{Ayumu Inoue}
\address{Department of Mathematics Education, Aichi University of Education, Kariya, Aichi 448-8542, Japan}
\email{ainoue@auecc.aichi-edu.ac.jp}

\subjclass[2010]{Primary 57M25; Secondary 51M04}
\keywords{quandle, coloring, Alexander polynomial, torus knot, PL trochoid, Euclidean geometry}

\begin{abstract}
The set consisting of all rotations of the Euclidean plane is equipped with a quandle structure.
We show that a knot is colorable by this quandle if and only if its Alexander polynomial has a root on the unit circle in $\mathbb{C}$.
Further we enumerate all non-trivial colorings of a torus knot diagram by the quandle using PL trochoids.
As an application of these results, we have the complete factorization of the Alexander polynomial of the torus knot.
\end{abstract}

\maketitle

\section{Introduction}
\label{sec:introduction}

A quandle, introduced by Joyce \cite{Joy1982} in 1982\footnote{The same notion was also introduced by Matveev \cite{Mat1982} in 1982, under the name of a distributive groupoid.}, is an algebraic system which is characterized by certain three axioms.
These three axioms are closely related to the three local moves of knot diagrams, known as the Reidemeister moves, respectively.
Therefore a quandle is considered as a minimal algebra for knots \cite{Nel2011}.
Indeed, it is known by Joyce \cite{Joy1982} and Matveev \cite{Mat1982} independently that quandles derived from knots are isomorphic if and only if the knots are equivalent.

On the other hand, a special quandle called a kei (\raisebox{-1.8pt}{\includegraphics[width=10pt]{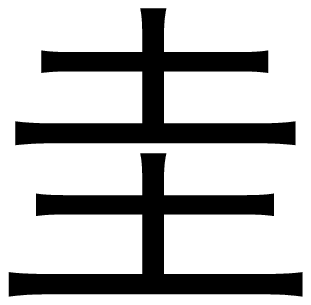}}\footnote{It is a Chinese character which is invariant under reflection about the centerline and invariant under swapping top and bottom.}) was studied by Takasaki \cite{Tak1943} in 1943.
The notion of a kei abstracts the behavior of symmetric transformations.
Takasaki said in his paper that, to investigate geometric symmetry, using a kei is more suitable than using a group.
A quandle also describes primitive symmetry well.
For example, each subset of a group closed under conjugations is equipped with a quandle structure, while it is no longer not equipped with a group structure in general.
This quandle still describes symmetry that the subset does as a part of the group.

How those two aspects of a quandle are related to each other?
In this paper, we focus on the quandle $\RotE{2}$ consisting of all rotations of the Euclidean plane, and investigate colorability of knots by $\RotE{2}$.
Of course, $\RotE{2}$ describes rotational symmetry of the Euclidean plane.
Colorability of knots is related to the first aspect.
We show in Section \ref{sec:necessary_and_sufficient_condition} that a knot is $\RotE{2}$-colorable if and only if its Alexander polynomial has a root on the unit circle in $\mathbb{C}$ (Theorem \ref{thm:necessary_and_sufficient_condition}).
Furthermore, in Section \ref{sec:torus_knot_and_PL_trochoid}, we enumerate all $\RotE{2}$-colorings of a torus knot diagram using PL trochoids (Theorem \ref{thm:torus_knot_and_PL_trochoid}).
As an application of these results, we have the complete factorization of the Alexander polynomial of the torus knot (Corollary \ref{cor:factorization_of_the_Alexander_polynomial_of_the_torus_knot}).

\section{Preliminaries}
\label{sec:preliminaries}

In this section, we review the definitions of a quandle and colorability of a knot by a quandle briefly.
We refer the reader to \cite{Car2012, Joy1982, Nel2011} for more details.
We further define the quandle consisting of all rotations of the Euclidean plane in which we are mainly interested in this paper.

A \emph{quandle} is a non-empty set $X$ equipped with a binary operation $\ast : X \times X \rightarrow X$ satisfying the following three axioms:
\begin{itemize}
 \item[(Q1)] For each $x \in X$, $x \ast x = x$.
 \item[(Q2)] For each $x \in X$, a map $\ast \> x : X \rightarrow X$ ($w \mapsto w \ast x$) is bijective.
 \item[(Q3)] For each $x, y, z \in X$, $(x \ast y) \ast z = (x \ast z) \ast (y \ast z)$.
\end{itemize}
A quandle is said to be a \emph{kei} (\raisebox{-1.8pt}{\includegraphics[width=10pt]{kei.eps}}) if each bijection $\ast \> x$ is involutive.
A typical example of a quandle is a subset $X$ of a group, which is closed under conjugations, with $x \ast y = y^{-1} x y$ for each $x, y \in X$.
We call it a \emph{conjugation quandle}.

Let $\RotE{2}$ be the set consisting of all rotations of the Euclidean plane $\mathbb{E}^{2}$.
Since it is a subset of the isometry group of $\mathbb{E}^{2}$ and is closed under conjugations, $\RotE{2}$ is equipped with a conjugation quandle structure.
This conjugation quandle $\RotE{2}$ is the quandle consisting of all rotations of the Euclidean plane.
We note that $\RotE{2}$ is isomorphic to a quandle $\mathbb{C} \times \mathrm{U}(1)$ whose binary operation $\ast$ is given by
\[
 (z, e^{\theta \sqrt{-1}}) \ast (w, e^{\eta \sqrt{-1}}) = ((z - w) e^{\eta \sqrt{-1}} + w, e^{\theta \sqrt{-1}}).
\]
Under the identification of the complex plane $\mathbb{C}$ with $\mathbb{E}^{2}$, an element $(z, e^{\theta \sqrt{-1}}) \in \mathbb{C} \times \mathrm{U}(1)$ corresponds to a rotation about $z$ by angle $\theta$.
In the remaining of this paper, we do not distinguish $\RotE{2}$ from $\mathbb{C} \times \mathrm{U}(1)$.

The axioms of a quandle are closely related to the Reidemeister moves of knot diagrams as follows.
A \emph{coloring} of a knot diagram $D$ by a quandle $X$ is a map $\{ \textrm{arcs of $D$} \} \rightarrow X$ satisfying the condition depicted in Figure \ref{fig:condition_of_arc_coloring} at each crossing.
We call an element of a quandle assigned to an arc by a coloring a \emph{color} of the arc.
Suppose $D^{\prime}$ is a knot diagram obtained from $D$ by a Reidemeister move.
Then, for each coloring $\mathscr{C}$ of $D$, we have a unique coloring of $D^{\prime}$ whose restriction to the arcs unrelated to the deformation coincides with the restriction of $\mathscr{C}$.
Indeed, the axioms (Q1), (Q2) and (Q3) of a quandle guarantee that we can perform the Reidemeister moves RI, RII and RIII fixing ends' colors respectively.
See Figure \ref{fig:correspondence_of_arc_colorings}.
Thus, for a fixed quandle, the number of all colorings gives rise to a knot invariant.
\begin{figure}[htbp]
\begin{center}
\includegraphics[scale=0.2]{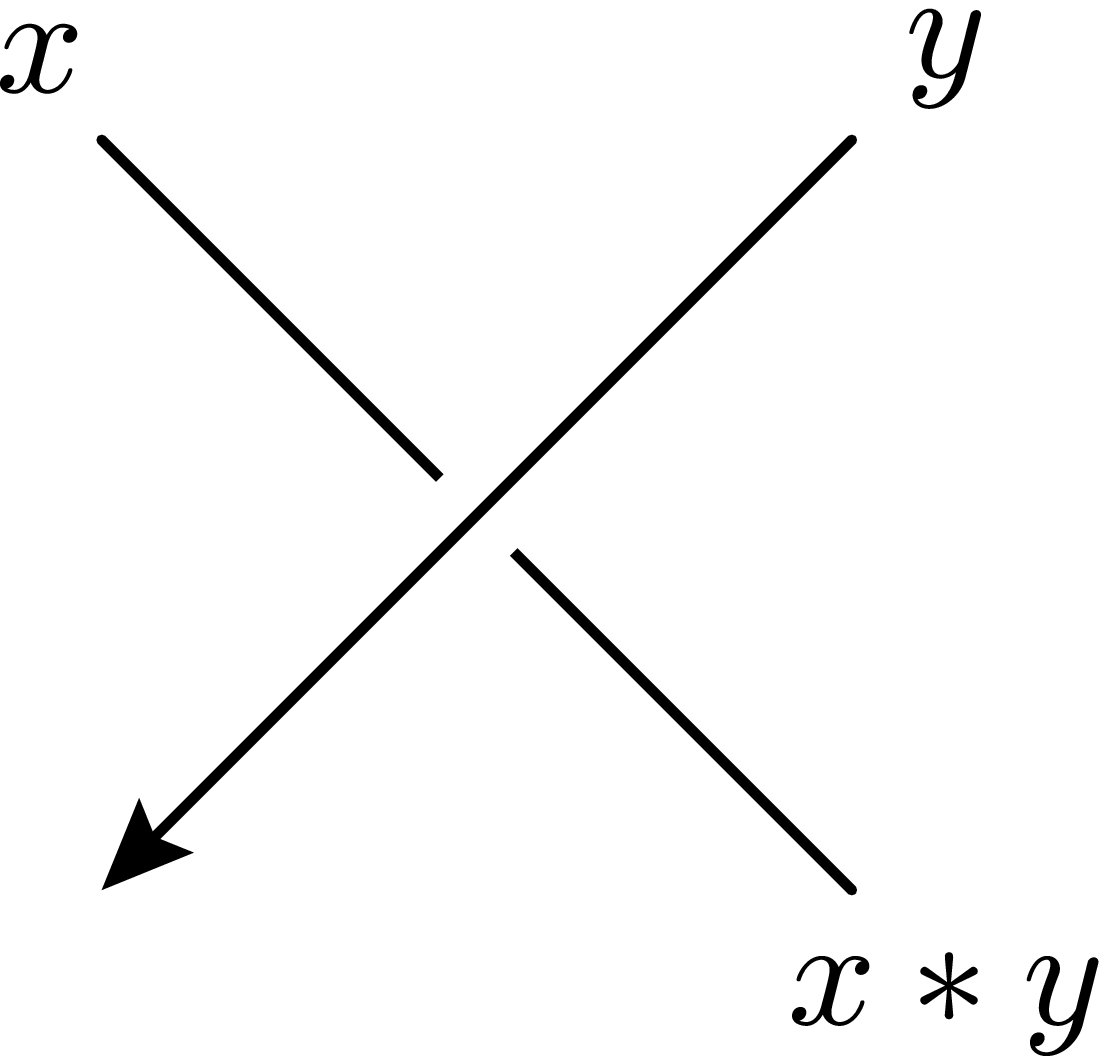}
\end{center}
\caption{}
\label{fig:condition_of_arc_coloring}
\end{figure}
\begin{figure}[htbp]
\begin{center}
\includegraphics[scale=0.2]{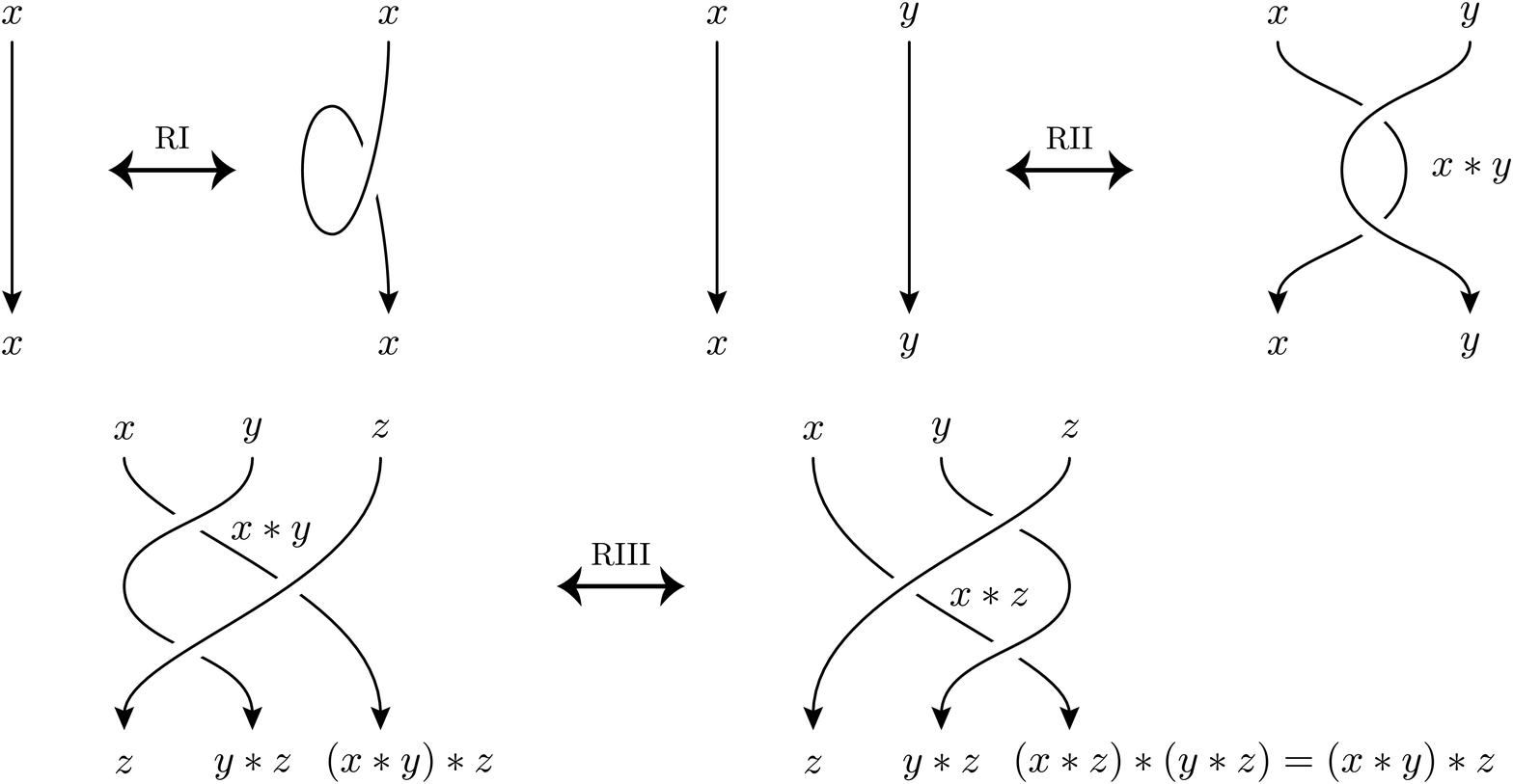}
\end{center}
\caption{}
\label{fig:correspondence_of_arc_colorings}
\end{figure}

A constant map from the set consisting of the arcs of a knot diagram to a quandle always satisfies the coloring condition.
We thus call this constant map a \emph{trivial coloring}.
For a quandle $X$, a knot $K$ is said to be $X$-\emph{colorable} if there is a non-trivial coloring of a diagram of $K$ by $X$.
We note that no non-trivial colorings are obtained from a trivial coloring by Reidemeister moves.

\section{Necessary and sufficient condition}
\label{sec:necessary_and_sufficient_condition}

Suppose $\RotE{2}$ is the quandle consisting of all rotations of the Euclidean plane, defined in the previous section.
Which knots are $\RotE{2}$-colorable?

\begin{example}[trefoil]
Consider a diagram of the right hand trefoil, depicted in the left-hand side of Figure \ref{fig:example_trefoil}, colored by $\RotE{2}$.
Here, $(z_{1}, e^{\theta_{1} \sqrt{-1}})$, $(z_{2}, e^{\theta_{2} \sqrt{-1}})$ and $(z_{3}, e^{\theta_{3} \sqrt{-1}})$ denote colors of the corresponding arcs.
We then have the following equations
\[
 (z_{1}, e^{\theta_{1} \sqrt{-1}}) = (z_{3}, e^{\theta_{3} \sqrt{-1}}) \ast (z_{2}, e^{\theta_{2} \sqrt{-1}}),
\]
\[
 (z_{2}, e^{\theta_{2} \sqrt{-1}}) = (z_{1}, e^{\theta_{1} \sqrt{-1}}) \ast (z_{3}, e^{\theta_{3} \sqrt{-1}}),
\]
\[
 (z_{3}, e^{\theta_{3} \sqrt{-1}}) = (z_{2}, e^{\theta_{2} \sqrt{-1}}) \ast (z_{1}, e^{\theta_{1} \sqrt{-1}})
\]
associated with the crossings.
The first equation requires that $\theta_{1} = \theta_{3}$ and $|z_{1} - z_{2}| = |z_{3} - z_{2}|$.
The other equations require $\theta_{2} = \theta_{1}$, $|z_{2} - z_{3}| = |z_{1} - z_{3}|$, $\theta_{3} = \theta_{2}$ and $|z_{3} - z_{1}| = |z_{2} - z_{1}|$.
These requests are satisfied if we take points $z_{1}, z_{2}, z_{3}$ so that the triangle $\triangle z_{1} z_{2} z_{3}$ is regular and set each $\theta_{i}$ to be $\pi / 3$.
See the right-hand side of Figure \ref{fig:example_trefoil}.
Therefore the right hand trefoil is $\RotE{2}$-colorable.
\begin{figure}[htbp]
\begin{center}
\includegraphics[scale=0.2]{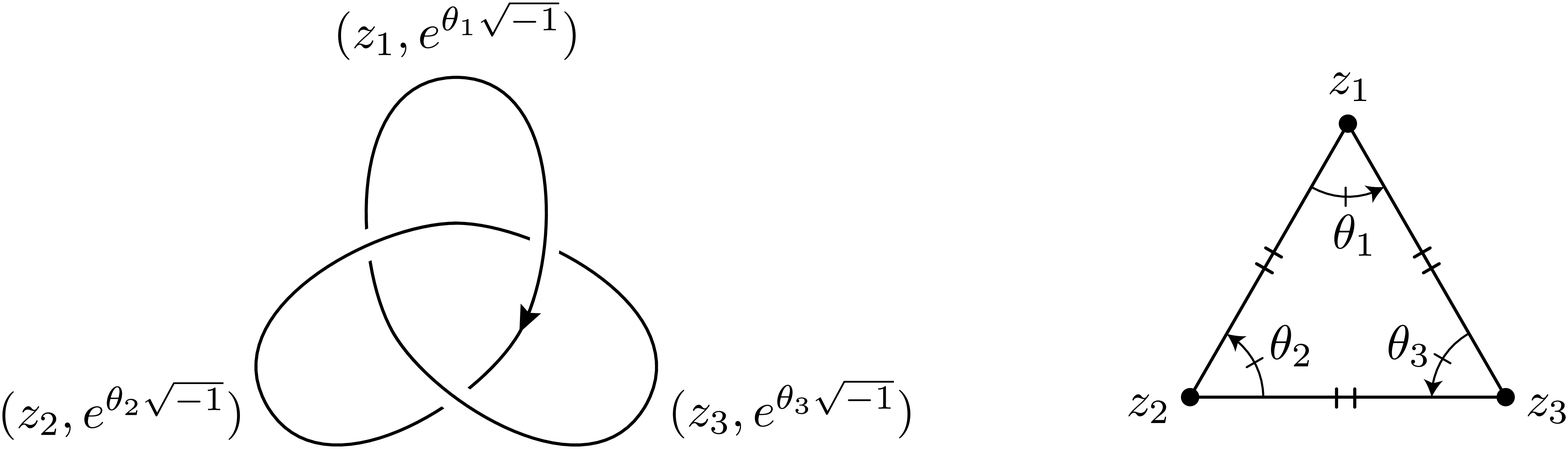}
\end{center}
\caption{}
\label{fig:example_trefoil}
\end{figure}
\end{example}

In the above example, it is required that all $\theta_{i}$ are the same.
The same request is obviously imposed for any knot diagram, because one can go around all arcs of a knot diagram passing under over arcs:

\begin{lemma}
\label{lem:same_angle}
For any knot diagram colored by $\RotE{2}$, the second components {\upshape (}i.e., rotation angles{\upshape )} of colors of the arcs are the same.
\end{lemma}

\begin{example}[figure eight knot]
\label{ex:figure_eight_knot}
We next consider a diagram of the figure eight knot, depicted in the left-hand side of Figure \ref{fig:example_figure_eight_knot}, colored by $\RotE{2}$.
In this case, we have the following equations
\[
 (z_{2}, e^{\theta \sqrt{-1}}) = (z_{1}, e^{\theta \sqrt{-1}}) \ast (z_{4}, e^{\theta \sqrt{-1}}), \enskip
 (z_{4}, e^{\theta \sqrt{-1}}) = (z_{3}, e^{\theta \sqrt{-1}}) \ast (z_{2}, e^{\theta \sqrt{-1}}),
\]
\[
 (z_{2}, e^{\theta \sqrt{-1}}) = (z_{3}, e^{\theta \sqrt{-1}}) \ast (z_{1}, e^{\theta \sqrt{-1}}), \enskip
 (z_{4}, e^{\theta \sqrt{-1}}) = (z_{1}, e^{\theta \sqrt{-1}}) \ast (z_{3}, e^{\theta \sqrt{-1}})
\]
associated with the crossings.
The first two equations require that $|z_{1} - z_{4}| = |z_{4} - z_{2}| = |z_{2} - z_{3}|$ and $z_{1} z_{4} \parallel z_{2} z_{3}$.
Further the last two equations require that $|z_{2} - z_{1}| = |z_{1} - z_{3}| = |z_{3} - z_{4}|$ and $z_{2} z_{1} \parallel z_{3} z_{4}$.
These requests are depicted in the right-hand side of Figure \ref{fig:example_figure_eight_knot} respectively.
There are no arrangements of points $z_{1}, z_{2}, z_{3}, z_{4}$ other than $z_{1} = z_{2} = z_{3} = z_{4}$ which satisfy the requests simultaneously.
Therefore the figure eight knot is not $\RotE{2}$-colorable.
\begin{figure}[htbp]
\begin{center}
\includegraphics[scale=0.2]{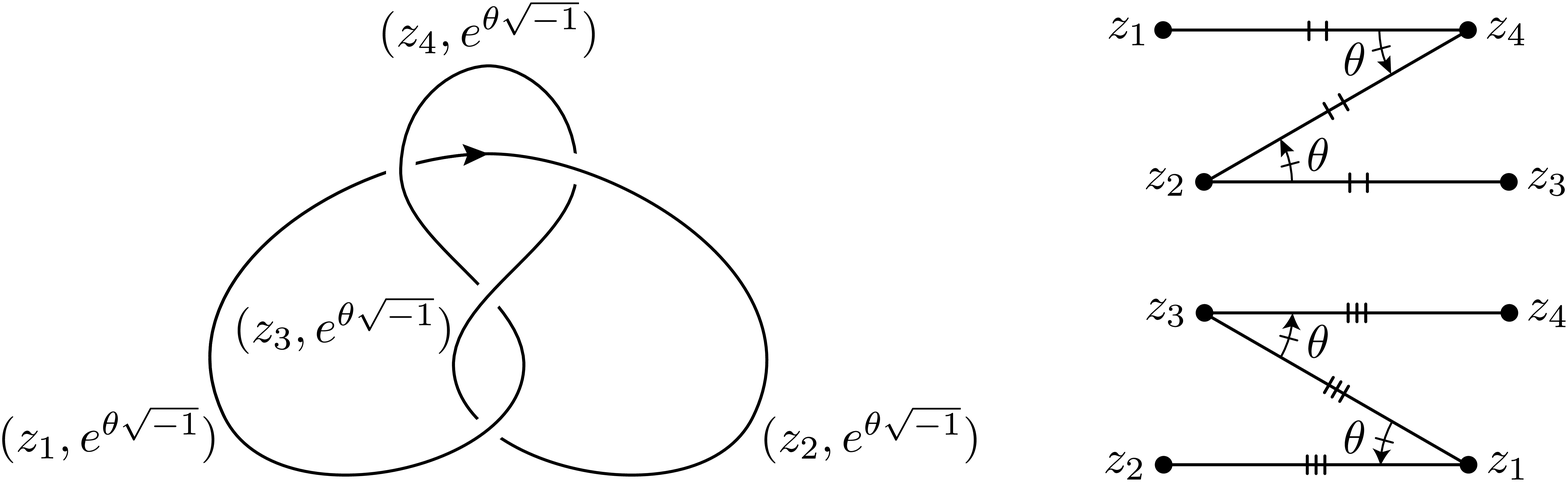}
\end{center}
\caption{}
\label{fig:example_figure_eight_knot}
\end{figure}
\end{example}

We note that the Alexander polynomial of the trefoil ($\RotE{2}$-colorable) is
\[
 \Delta_{3_{1}}(t) = t^{2} - t + 1 = \left( t - \exp \left( \dfrac{\pi \sqrt{-1}}{3} \right) \right) \left( t - \exp \left(- \dfrac{\pi \sqrt{-1}}{3} \right) \right),
\]
and that of the figure eight knot (not $\RotE{2}$-colorable) is
\[
 \Delta_{4_{1}}(t) = t^{2} - 3 t + 1 = \left( t - \dfrac{3 - \sqrt{5}}{2} \right) \left( t - \dfrac{3 + \sqrt{5}}{2} \right).
\]

\begin{theorem}
\label{thm:necessary_and_sufficient_condition}
A knot $K$ is $\RotE{2}$-colorable if and only if its Alexander polynomial $\Delta_{K}(t)$ has a root on the unit circle in $\mathbb{C}$.
More precisely, there is a non-trivial coloring of a diagram of $K$ by $\RotE{2}$, whose rotation angles are $\theta$, if and only if $\Delta_{K}(e^{\theta \sqrt{-1}}) = 0$.
\end{theorem}

\begin{proof}
Let $D$ be a diagram of $K$, $a_{1}, a_{2}, \dots, a_{n}$ the arcs of $D$, $c_{i}$ the crossing of $D$ from which $a_{i}$ starts, and $\varepsilon_{i}$ the sign of $c_{i}$.
Suppose $a_{k_{i}}$ and $a_{l_{i}}$ denote the under arc other than $a_{i}$ and the over arc related to $c_{i}$ respectively.
See figure \ref{fig:crossings}.
\begin{figure}[htbp]
\begin{center}
\includegraphics[scale=0.2]{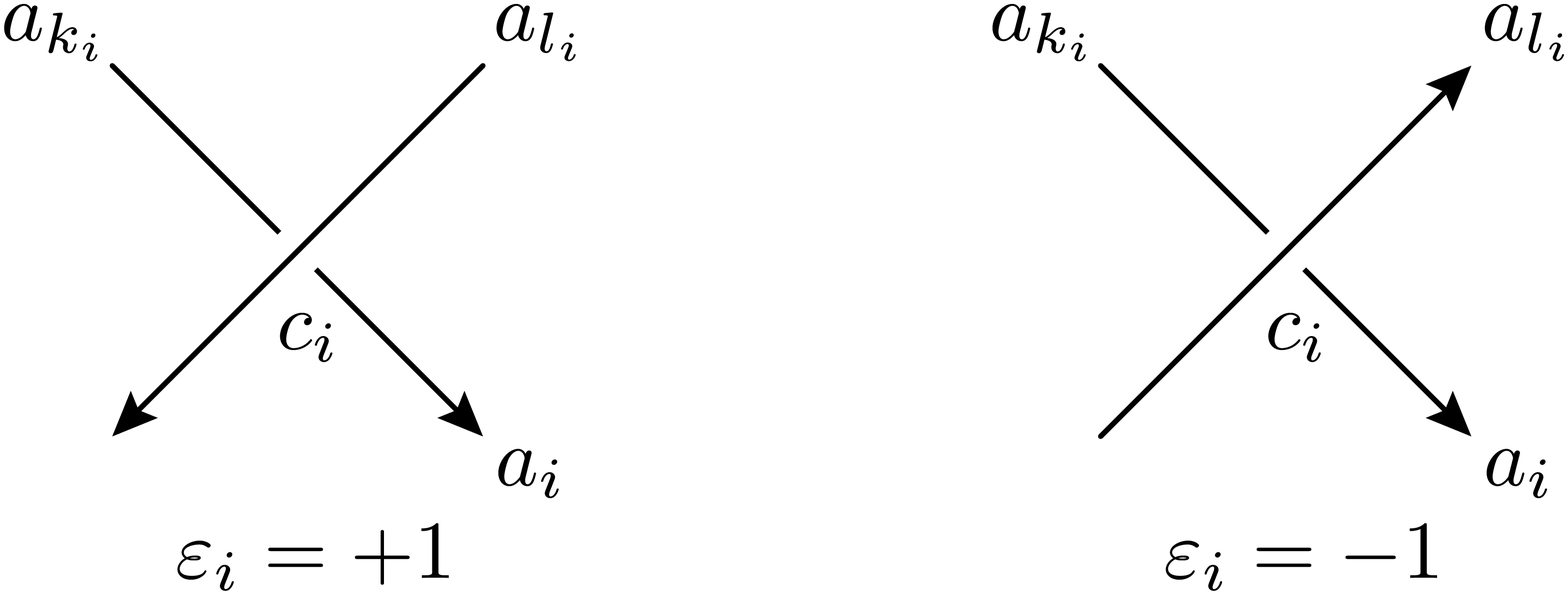}
\end{center}
\caption{}
\label{fig:crossings}
\end{figure}

For a fixed $\theta \in \mathbb{R}$, consider an $n \times n$ matrix $X_{\theta}$ whose $(k_{i}, i)$ entry is $\exp(\varepsilon_{i} \theta \sqrt{-1})$, $(l_{i}, i)$ entry is $1 - \exp(\varepsilon_{i} \theta \sqrt{-1})$, $(i, i)$ entry is $-1$, and otherwise is $0$.
Then a map $a_{i} \mapsto (z_{i}, \theta)$ is a $\RotE{2}$-coloring of $D$ if and only if $(z_{1}, z_{2}, \dots, z_{n}) X_{\theta}$ is equal to $(0, 0, \dots, 0)$.
Since we always have trivial colorings, the rank of $X_{\theta}$ is at most $n - 1$.
There is a non-trivial $\RotE{2}$-coloring, whose rotation angles are $\theta$, if and only if the rank of $X_{\theta}$ is less than $n - 1$.

On the other hand, let $A_{D}$ be the Alexander matrix for $K$ derived from the Wirtinger presentation of the knot group related to $D$ with Fox's free differential calculus and the abelization of the knot group.
See \cite[Theorem 9.10]{BZ2003}, for example.
Actually, $A_{D}$ is an $n \times n$ matrix whose $(k_{i}, i)$ entry is $t^{\varepsilon_{i}}$, $(l_{i}, i)$ entry is $1 - t^{\varepsilon_{i}}$, $(i, i)$ entry is $-1$, and otherwise is $0$.
Since $X_{\theta} = A_{D}|_{t = \exp(\theta \sqrt{-1})}$ and $\Delta_{K}(t)$ is the greatest common divisor of all $(n - 1) \times (n - 1)$ minors of $A_{D}$, the rank of $X_{\theta}$ is less than $n - 1$ if and only if $e^{\theta \sqrt{-1}}$ is a root of $\Delta_{K}(t)$.
\end{proof}

\begin{remark}
\label{rem:rank}
For each $i$ ($0 \leq i \leq n - 2$), the greatest common divisor of all $(n - i - 1) \times (n - i - 1)$ minors of $A_{D}$ is called the $i$-th Alexander polynomial of $K$ and is denoted by $\Delta_{K}^{(i)}(t)$.
Thus $\Delta_{K}^{(0)}(t) = \Delta_{K}(t)$.
We further let $\Delta_{K}^{(n-1)}(t) = 1$.
Since the elementary divisors of $A_{D}$ are $0$ and $e_{K}^{(i)}(t) = \Delta_{K}^{(i)}(t) / \Delta_{K}^{(i+1)}(t)$ ($0 \leq i \leq n - 2$), we have the following equation:
\[
 \mathrm{rank} \hskip 0.1em X_{\theta} = n - 1 - |\{ i \mid e_{K}^{(i)}(e^{\theta \sqrt{-1}}) = 0 \enskip (0 \leq i \leq n - 2) \}|.
\]
\end{remark}

\begin{remark}
\label{rem:similarity}
Assume that the rank of $X_{\theta}$ is $n - 2$.
Let $a_{i} \mapsto (z_{i}, \theta)$ be a non-trivial $\RotE{2}$-coloring of $D$.
Then, by the assumption, any non-trivial $\RotE{2}$-coloring, whose rotation angles are $\theta$, is given by $a_{i} \mapsto (\alpha z_{i} + \beta, \theta)$ with some $\alpha \in \mathbb{C} \setminus \{ 0 \}$ and $\beta \in \mathbb{C}$.
We note that corrections of points $\{ \alpha z_{1} + \beta, \alpha z_{2} + \beta, \dots, \alpha z_{n} + \beta \}$ are related to each other by orientation preserving similarities.
Therefore a non-trivial $\RotE{2}$-coloring, whose rotation angles are $\theta$, is unique up to orientation preserving similarities in this case.
\end{remark}

\begin{remark}
Let $G$ be the group consisting of all orientation preserving similarities of $\mathbb{C}$.
Burde and Zieschang showed that there is a non-trivial representation of the knot group to $G$ mapping a positive meridian to a rotation if and only if the Alexander polynomial of the knot has a root on the unit circle in $\mathbb{C}$ (Proposition 14.5 in \cite{BZ2003}).
Since $\RotE{2}$ is also a subset of $G$ closed under conjugations, in the light of Lemmas 3.3 and 3.5 in \cite{Ino2010}, there is a one-to-one correspondence between a $\RotE{2}$-coloring and a representation mapping a positive meridian to a rotation.
Therefore Theorem \ref{thm:necessary_and_sufficient_condition} is essentially no different from Proposition 14.5 in \cite{BZ2003}.
\end{remark}

\section{Torus knot and PL trochoid}
\label{sec:torus_knot_and_PL_trochoid}

The purpose of this section is to enumerate all non-trivial $\RotE{2}$-colorings of a torus knot diagram.
Since the Alexander polynomial of the $(p, q)$-torus knot is
\[
 \Delta_{T(p, q)}(t) = \dfrac{(t^{pq} - 1) (t - 1)}{(t^{p} - 1) (t^{q} - 1)},
\]
in the light of Theorem \ref{thm:necessary_and_sufficient_condition}, the $(p, q)$-torus knot is $\RotE{2}$-colorable.
Indeed, each root of $\Delta_{T(p, q)}(t)$ is obviously a root of unity.

In order to achieve our goal, we need the following notations.
Let $\Pi(m)$ be a convex regular $m$-gon in $\mathbb{C}$ ($m \geq 2$) and $v^{m}_{0}, v^{m}_{1}, \dots, v^{m}_{m-1}$ the vertices of $\Pi(m)$ in counterclockwise order.
For each $k$ and $i$ ($1 \leq k \leq m - 1$, $0 \leq i \leq m - 1$), we suppose that $v^{m, k}_{i}$ denotes the vertex $v^{m}_{[ik]}$, where $[r]$ denotes the integer satisfying $0 \leq [r] \leq m - 1$ and $[r] \equiv r \pmod m$  for each $r \in \mathbb{Z}$.
Let $\Pi(m, k)$ be the regular polygon in $\mathbb{C}$ obtained by joining vertices $v^{m, k}_{[i]}$ and $v^{m, k}_{[i+1]}$ for each $i$.
We note that some vertices and edges of $\Pi(m, k)$ are overlapped if $m$ and $k$ are not coprime.
Further $\Pi(m, m - k)$ is the mirror image of $\Pi(m, k)$.
Various $\Pi(5, \bullet)$ and $\Pi(6, \bullet)$ are depicted in Figure \ref{fig:regular_polygons}, for example.
It is easy to see that
\[
 \angle v^{m, k}_{[i]} v^{m, k}_{[i+1]} v^{m, k}_{[i+2]} = \left| \dfrac{(m - 2 k) \pi}{m} \right|
\]
for each $i$.
\begin{figure}[htbp]
\begin{center}
\includegraphics[scale=0.2]{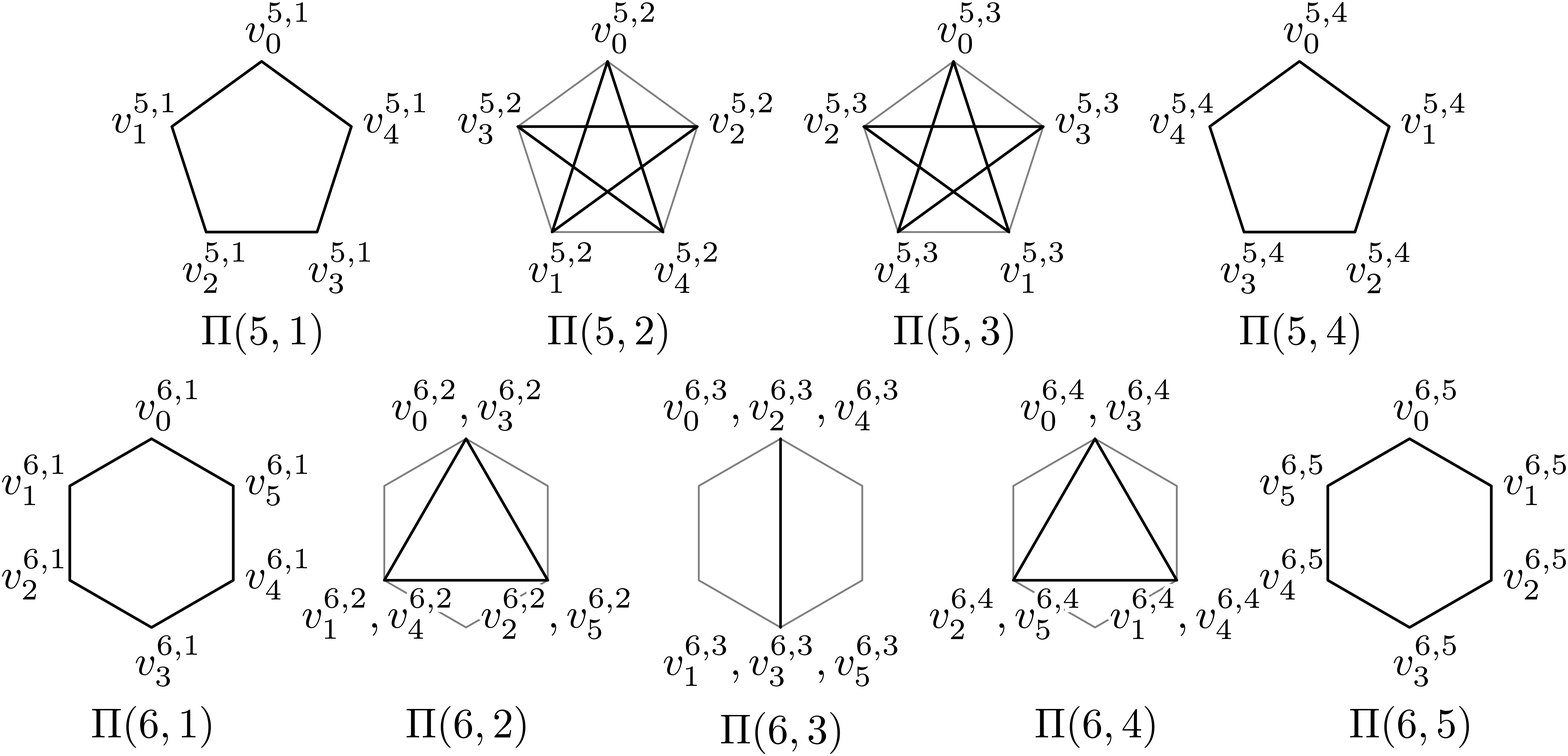}
\end{center}
\caption{}
\label{fig:regular_polygons}
\end{figure}

Suppose $m$, $n$, $k$ and $l$ are integers with $m, n \geq 2$, $1 \leq k \leq m - 1$ and $1 \leq l \leq n - 1$.
We define a real number $\theta(m, k; n, l)$ by
\[
 \theta(m, k; n, l) = \left( \dfrac{m - 2 k}{m} - \dfrac{n - 2 l}{n} \right) \pi.
\]
Consider two regular polygons $\Pi(m, k)$ and $\Pi(n, l)$ whose side lengths are the same.
We arrange $\Pi(m, k)$ and $\Pi(n, l)$ so that $v^{m, k}_{[0]} = v^{n, l}_{[0]}$ and $v^{m, k}_{[1]} = v^{n, l}_{[1]}$.
Fixing the position of $\Pi(n, l)$, we first rotate $\Pi(m, k)$ about $v^{m, k}_{[1]} = v^{n, l}_{[1]}$ by angle $\theta(m, k; n, l)$.
Then $v^{m, k}_{[2]}$ meets with $v^{n, l}_{[2]}$.
Thus we next rotate $\Pi(m, k)$ about $v^{m, k}_{[2]} = v^{n, l}_{[2]}$ instead of $v^{m, k}_{[1]} = v^{n, l}_{[1]}$ by angle $\theta(m, k; n, l)$, which produces that $v^{m, k}_{[3]} = v^{n, l}_{[3]}$.
We continue this process in the same manner.
The $i$-th step is the rotation of $\Pi(m, k)$ about $v^{m, k}_{[i]} = v^{n, l}_{[i]}$ by angle $\theta(m, k; n, l)$.
We call this motion of $\Pi(m, k)$ related to $\Pi(n, l)$ the \emph{$(m, k; n, l)$-trochoid}.
Figure \ref{fig:4_1_3_1_trochoid} illustrates the $(4, 1; 3, 1)$-trochoid, for example.
\begin{figure}[htbp]
\begin{center}
\includegraphics[scale=0.2]{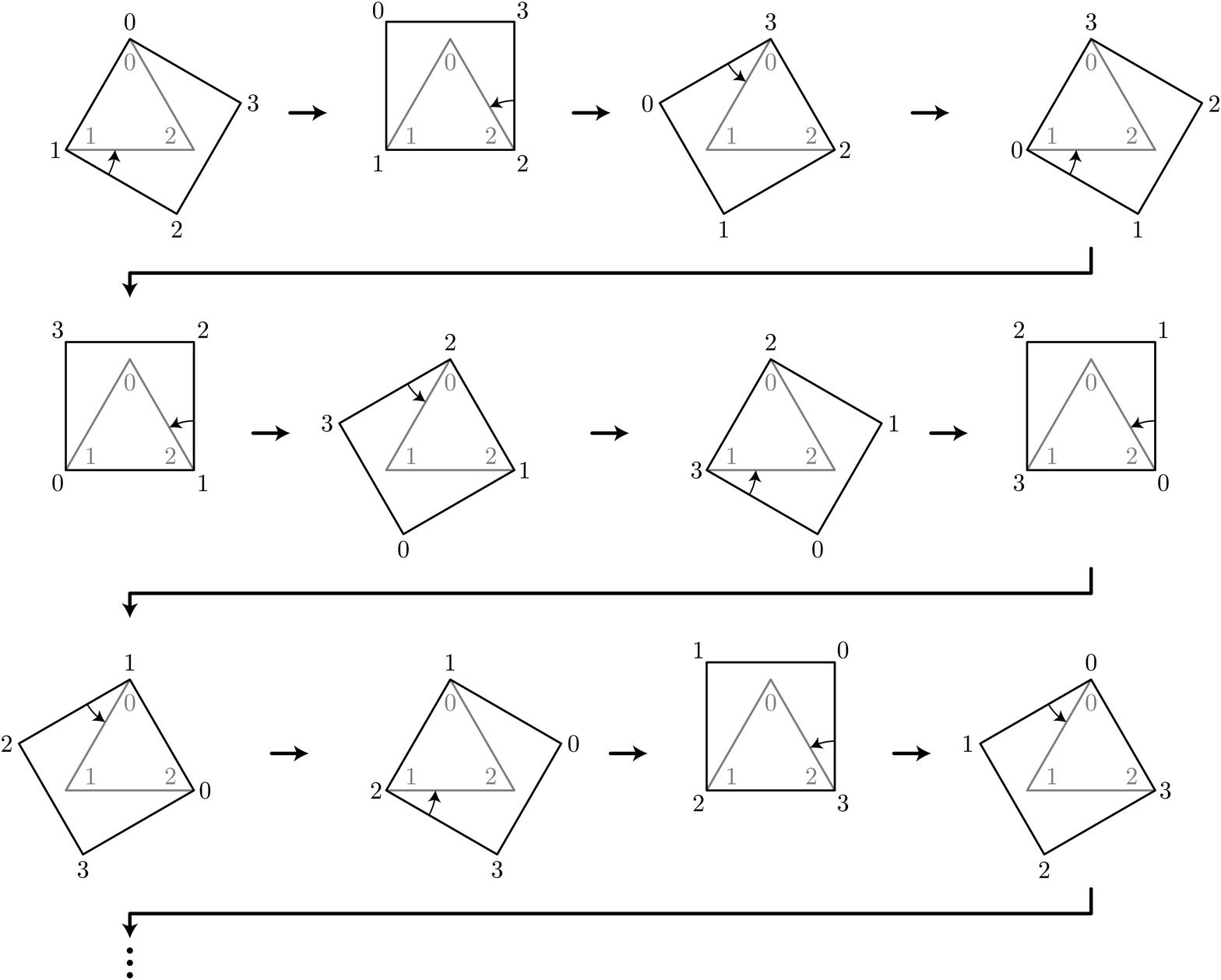}
\end{center}
\caption{}
\label{fig:4_1_3_1_trochoid}
\end{figure}

\begin{theorem}
\label{thm:torus_knot_and_PL_trochoid}
Let $p$ and $q$ be non-zero coprime integers and $D(p, q)$ a diagram of the $(p, q)$-torus knot depicted in Figure \ref{fig:torus_knot_diagram}.
Then, for each $k$ and $l$ {\upshape (}$1 \leq k \leq |p| - 1$, $1 \leq l \leq |q| - 1${\upshape )}, there is a non-trivial $\RotE{2}$-coloring of $D(p, q)$ derived from the $(|p|, k; |q|, l)$-trochoid.
Rotation angles of the non-trivial coloring are $\theta(|p|, k; |q|, l)$.
\begin{figure}[htbp]
\begin{center}
\includegraphics[scale=0.163]{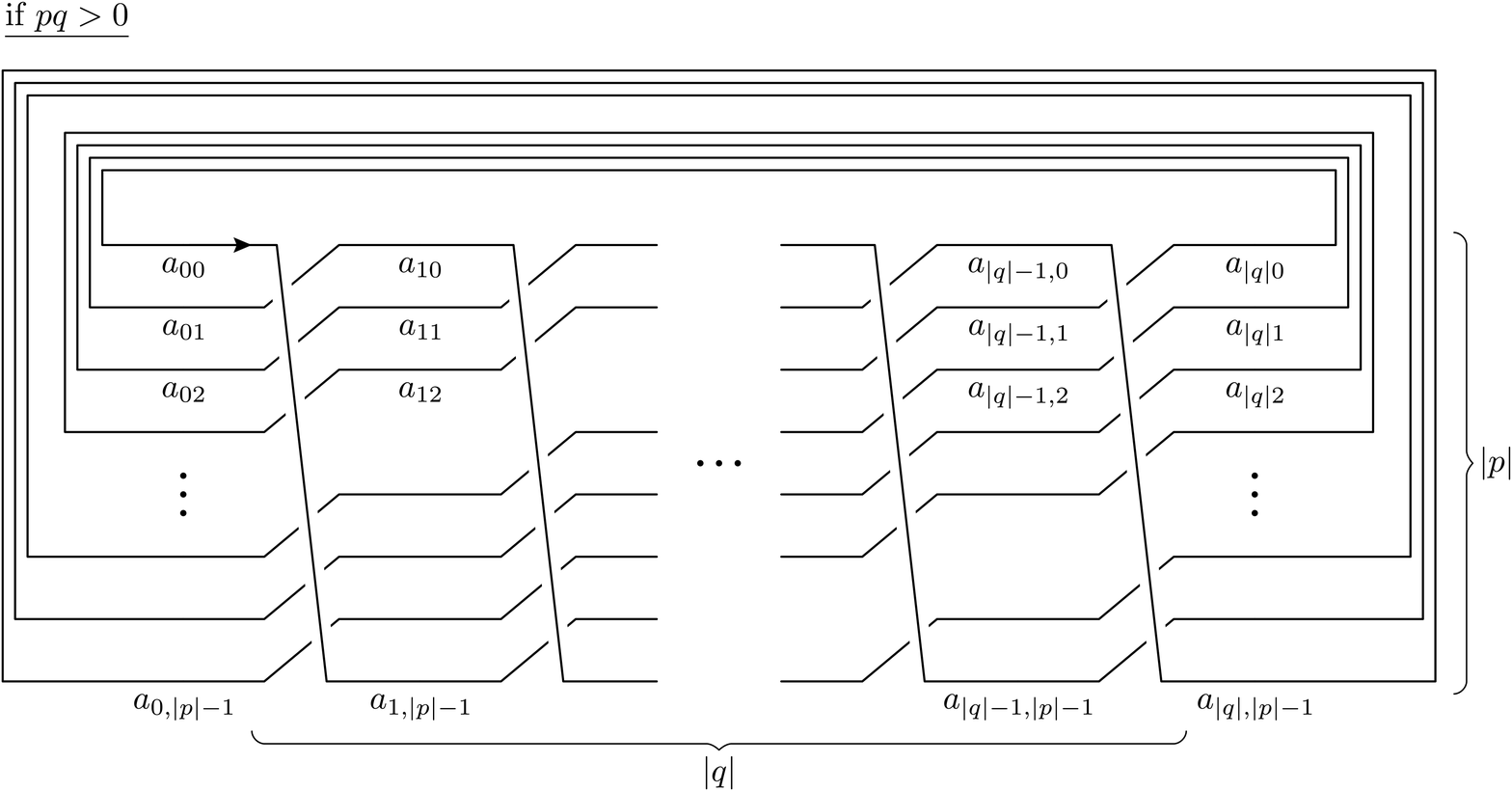}

\vspace{3ex}

\includegraphics[scale=0.163]{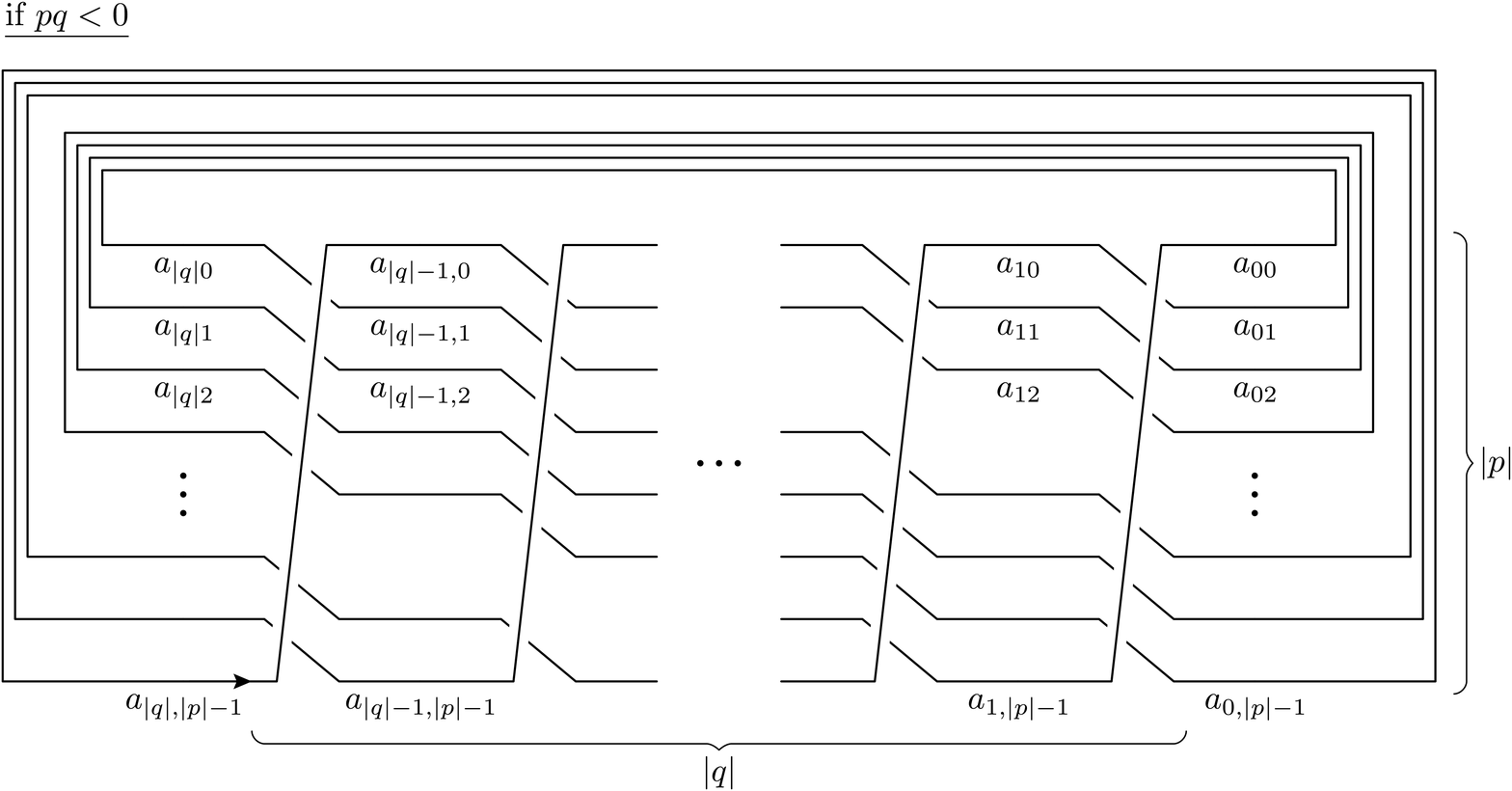}
\end{center}
\caption{}
\label{fig:torus_knot_diagram}
\end{figure}
\end{theorem}

\begin{proof}
First, we let $a_{ij}$ denote the arcs of $D(p, q)$ as depicted in Figure \ref{fig:torus_knot_diagram}, although $a_{i0}$ and $a_{i+1, |p|-1}$ (resp. $a_{|q|j}$ and $a_{0j}$) mark the same arc.

For each $i$ and $j$ ($0 \leq i \leq |q|$, $0 \leq j \leq |p| - 1$), let $z_{ij} \in \mathbb{C}$ be the coordinate of the vertex $v^{|p|, k}_{[i+j+1]}$ of $\Pi(|p|, k)$ after the $i$-th step of the $(|p|, k; |q|, l)$-trochoid.
Here, the $0$-th step means the arrangement of $\Pi(|p|, k)$ and $\Pi(|q|, l)$ into the initial position.
These $z_{ij}$ for the $(4, 1; 3, 1)$-trochoid are depicted in Figure \ref{fig:colors_derived_from_4_1_3_1_trochoid}, for example.

By definition, $z_{i0}$ is equal to $z_{i+1, |p| - 1}$ for each $i$ ($0 \leq i \leq |q| - 1$).
Further the rotation about $z_{i0}$ by angle $\theta(|p|, k; |q|, l)$ sends $z_{ij}$ to $z_{i+1, j-1}$ for each $i$ and $j$ ($0 \leq i \leq |q| - 1$, $1 \leq j \leq |p| - 1$).
Since $\Pi(|p|, k)$ returns to the initial position but is rotated by angle $(|p| l - |q| k) \cdot (2 \pi / |p|)$ after the $|q|$-th step, we have that $v^{|p|, k}_{[|q|]} = v^{|q|, l}_{[0]}$ and $v^{|p|, k}_{[|q| + 1]} = v^{|q|, l}_{[1]}$.
Therefore $z_{|q|j}$ is equal to $z_{0j}$ for each $j$ ($0 \leq j \leq |p| - 1$).
By the above argument, we have a non-trivial $\RotE{2}$-coloring $a_{ij} \mapsto (z_{ij}, \theta(|p|, k; |q|, l))$ of $D(p, q)$.
\begin{figure}[htbp]
\begin{center}
\includegraphics[scale=0.2]{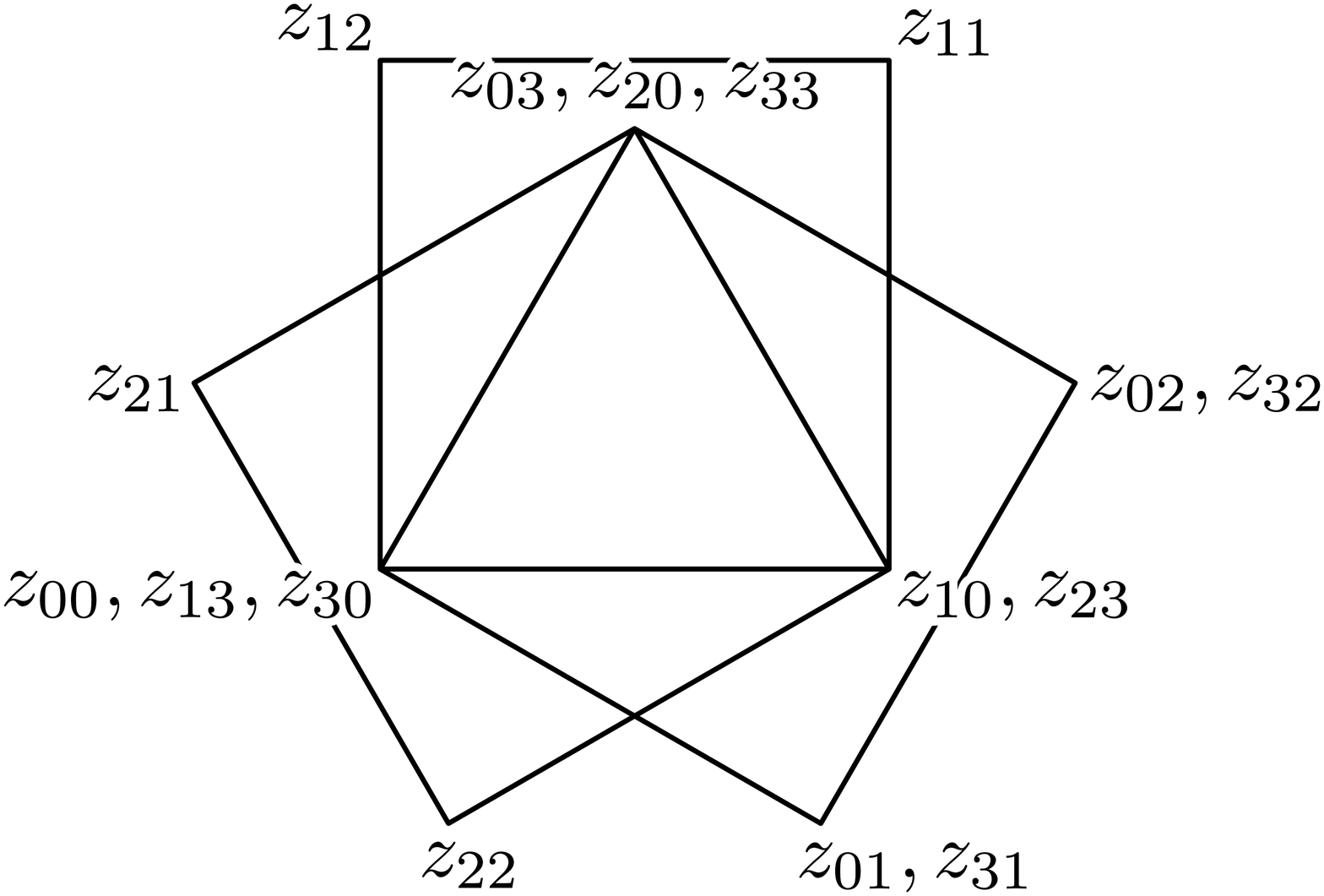}
\end{center}
\caption{}
\label{fig:colors_derived_from_4_1_3_1_trochoid}
\end{figure}
\end{proof}

We note that $(|p| - 1) (|q| - 1)$ complex numbers $e^{\theta(|p|, k; |q|, l) \sqrt{-1}}$ ($1 \leq k \leq |p| - 1$, $1 \leq l \leq |q| - 1$) are mutually distinct if $p$ and $q$ are coprime.
Further the span of the Alexander polynomial $\Delta_{T(p, q)}(t)$ of the $(p, q)$-torus knot is $(|p| - 1) (|q| - 1)$.
Therefore, in the light of Theorems \ref{thm:necessary_and_sufficient_condition} and \ref{thm:torus_knot_and_PL_trochoid}, we have the following corollary:

\begin{corollary}
\label{cor:factorization_of_the_Alexander_polynomial_of_the_torus_knot}
The Alexander polynomial of the $(p, q)$-torus knot is factorized with a certain $r \in \mathbb{Z}$ as follows:
\[
 \Delta_{T(p, q)}(t) = t^{r} \prod_{k = 1}^{|p| - 1} \prod_{l = 1}^{|q| - 1} (t - e^{\theta(|p|, k; |q|, l) \sqrt{-1}}).
\]
\end{corollary}

Suppose $X_{\theta}$ is the $n \times n$ matrix, defined in the proof of Theorem \ref{thm:necessary_and_sufficient_condition}, derived from $D(p, q)$ ($n = |q|(|p| - 1)$).
In the light of Corollary \ref{cor:factorization_of_the_Alexander_polynomial_of_the_torus_knot} and Remark \ref{rem:rank}, the rank of $X_{\theta}$ is $n - 2$ if $\theta = \theta(|p|, k; |q|, l)$ with some $k$ and $l$, and is $n - 1$ otherwise.
Therefore, in the light of Remark \ref{rem:similarity}, all non-trivial $\RotE{2}$-colorings of the $(p, q)$-torus knot diagram $D(p, q)$ are essentially enumerated by Theorem \ref{thm:torus_knot_and_PL_trochoid}.

\section{Related topics}
\label{sec:related_topics}

We close this paper with a discussion of some related topics.

Let $\RotE{n}$ be the set consisting of all rotations of the $n$-dimensional Euclidean space ($n \geq 3$).
Obviously, $\RotE{n}$ is equipped with a conjugation quandle structure.
Although it is not $\RotE{2}$-colorable as we have seen in Example \ref{ex:figure_eight_knot}, the figure eight knot is $\RotE{3}$-colorable.
Indeed, it is known that the figure eight knot is colorable by the \emph{tetrahedral quandle} which is a subquandle of $\RotE{3}$.
See \cite{Car2012}, for example.
Thus it seems to be natural that we have the following questions:

\begin{question}
For a non-trivial knot $K$ whose Alexander polynomial $\Delta_{K}(t)$ has no roots on the unit circle in $\mathbb{C}$, which $n$ is the minimum number so that $K$ is $\RotE{n}$-colorable?
\end{question}

\begin{question}
Is there a non-trivial knot which is not $\RotE{n}$-colorable for any $n$?
\end{question}

\begin{question}
For $n \geq 3$, what is meaning of $\RotE{n}$-colorability?
For example, are there relationships between $\RotE{n}$-colorability and other knot invariants?
\end{question}

The set consisting of all reflections of the Euclidean space is obviously equipped with a conjugation quandle structure.
Further the sets consisting of all rotations or reflections of spherical or hyperbolic spaces are also equipped with conjugation quandle structures.

\begin{question}
Which knots are colorable by these ``geometric'' quandles?
\end{question}

\begin{question}
What is meaning of colorability by the quandles?
\end{question}

A coloring of a (multi-component) link diagram by a quandle is defined in a similar way for a knot diagram.
A map from the arcs of a link diagram to $\RotE{2}$ sending the arcs of the $i$-th component to a fixed $(w, e^{\theta_{i} \sqrt{-1}})$ satisfies the coloring condition, although $\theta_{i} \neq \theta_{j}$ for some $i \neq j$.
Thus a link is said to be $\RotE{2}$-colorable if there is a $\RotE{2}$-coloring of its diagram other than all the above.

\begin{question}
Which links are $\RotE{2}$-colorable?
Further, which links are colorable by the other quandles?
\end{question}

\begin{question}
What is meaning of colorability of links by $\RotE{2}$ and the others?
\end{question}

\bibliographystyle{amsplain}

\end{document}